\documentclass[11pt]{article} 

\usepackage[utf8]{inputenc} 

\usepackage{geometry} 
\usepackage{graphicx} 

\usepackage{array} 
\usepackage{paralist} 
\usepackage{verbatim} 
\usepackage{subfig} 
\usepackage{amsthm}
\usepackage{amsmath}
\usepackage{amsfonts}
\usepackage{amssymb}
\usepackage{cite}
\usepackage{cleveref}

\newtheorem*{theorem*}{Theorem}
\newtheorem{theorem}{Theorem}[section]
\newtheorem{lemma}{Lemma}[section]
\newtheorem{corollary}{Corollary}[section]

\title{On certain tilting modules for $SL_2$}
\author{Samuel Martin}
  
\begin{document}
\maketitle
       
\begin{abstract}
	We give a complete picture of when the tensor product of an induced module and a Weyl module is a tilting module for the algebraic group $SL_2$ over an algebraically closed field of characteristic $p$. Whilst the result is recursive by nature, we give an explicit statement in terms of the $p$-adic expansions of the highest weight of each module.
\end{abstract}

\section{Introduction}

In this article we investigate the tensor product $\nabla(r)\otimes\Delta(s)$ for the group $G = SL_2(k)$, where $k$ is an algebraically closed field of characteristic $p>0$. Similar tensor products for $SL_2(k)$ have been studied before, in particular the product $L(r)\otimes L(s)$ by Doty and Henke in 2005\cite{Doty-Henke}. Motivated by their results utilising tilting modules, this paper describes exactly when the product $\nabla(r)\otimes\Delta(s)$ is a tilting module.
\\
\\
By an argument of Donkin given in \cite[Lemma~3.3]{Parker}, it's known already that when $|r - s | \leq 1$ the module $\nabla(r)\otimes\Delta(s)$ is tilting. Some special cases are also known, for example the tensor product of Steinberg modules $\nabla(p^n - 1)\otimes\Delta(p^m - 1)$ is tilting, since $\nabla(p^k - 1) = \Delta(p^k - 1)$ for all $k \in \mathbb{N}$, as is the tensor product $\nabla(a)\otimes\Delta(b)$ for $a,b \in \{0,\ldots, p- 1\}$.

\subsection{Terminology}
Before beginning, we will fix some terminology. Throughout, $k$ will be an algebraically closed field of characteristic $p > 0$, and $G$ will be the affine algebraic group $SL_2(k)$. Pick the Borel subgroup $B$ consisting of lower triangular matrices and containing the maximal torus $T$ of diagonal matrices. Let $X(T)$ be the weight lattice, which we associate with $\mathbb{Z}$ in the usual manner. Under this association the set of dominant weights $X^{+}$ corresponds to the set $\mathbb{N}\cup\{0\}$.
\\
\\
Whenever we refer to a module, we will always mean a rational $kG$-module, where $kG$ is the group algebra. Let $F:G\longrightarrow G$ denote the usual Frobenius morphism, and denote by $G_1$ its kernel. For any module $V$, we will denote by $V^F$ the Frobenius twist of $V$.
\\
\\
Let $k_r$ be the one dimensional $B$ module on which $T$ acts via $r \in \mathbb{Z}$, and let $\nabla(r)$ be the induced module $\text{Ind}_B^G(k_r)$. Then $\nabla(r)$ is finite dimensional and is non-zero only when $r$ is dominant, i.e. $r \geq 0$. It is well known that $\nabla(r) = S^rE$, the $r^{\text{th}}$ symmetric power of the natural module $E$. Let $\Delta(r)$ be the Weyl module given by $\Delta(r) = \nabla(r)^{*}$. By a tilting module we mean a module which has both a $\nabla$-filtration (or good filtration) and a $\Delta$-filtration (or Weyl filtration) as defined in \cite{Donkin}. We will denote by $T(r)$ the unique indecomposable tilting module of highest weight $r \in X^{+}$.
\\
\\
We will make use of the character $\text{Ch}(V)$ of a module $V$. This is given by
$$\text{Ch}(V) = \sum_{r \in X(T)}(\text{dim}V^r)x^r$$
inside the ring $\mathbb{Z}[x, x^{-1}]$ of Laurent polynomials, where $V^r$ is the $r$ weight space of $V$. We will write $\chi(r)$ for $\text{Ch}(\nabla(r)) = \text{Ch}(\Delta(r))$ and note that $\chi(1) = x + x^{-1}$. From the action of the Weyl group on each weight space, we have in fact that $\text{Ch}(V) \in \mathbb{Z}[\chi(1)]$, which is a unique factorization domain. 
\\
\\
The objects of interest in this article are the modules $\nabla(r)\otimes\Delta(s)$, for dominant weights $r$ and $s$. The character of these modules is given by the well known Clebsch-Gordan formula (assuming $r \geq s$)

$$ \text{Ch}(\nabla(r)\otimes\Delta(s)) = \chi(r)\chi(s) = \sum_{i=0}^s \chi(r + s - 2i).$$

Furthermore, for $r \geq s-1$, the module $\nabla(r)\otimes\Delta(s)$ has a $\nabla$-filtration (see \cite[Lemma~3.3]{Parker}) with sections given by those weights occurring in the character.

\subsection{Tilting Modules}
We will now give several useful results which will be used in the proceeding sections. In particular we will make extensive use of the following well known result, for which we have outlined a proof for the reader's convenience.
\begin{theorem}\label{Martin}
	There exists a short exact sequence given by
	$$ 0 \longrightarrow \nabla(r-1) \longrightarrow \nabla(r)\otimes E \longrightarrow \nabla(r+1) \longrightarrow 0,$$
and this is split if and only if $p$ does not divide $r+1$.
\end{theorem}
\begin{proof}
	That the sequence exists is clear by considering the $\nabla$-filtration of $\nabla(r)\otimes E = \nabla(r)\otimes\Delta(1)$. If $p$ does not divide $r+1$, the result follows by considering the blocks (see \cite[II.7.1]{Jantzen}) for $SL_2(k)$. On the other hand, if $p$ does divide $r+1$, then the module $E\otimes\nabla(r)$ is projective as a $G_1$-module, while neither $\nabla(r-1)$ nor $\nabla(r+1)$ are, so the sequence cannot be split.
\end{proof}

The next result extends  \cite[Lemma~3.3]{Parker}.

\begin{lemma}\label{tensorE}
	
If $r, s \in \{np - 1,\ np,\ np + 1, \ldots,\ (n+1)p - 1\}$ for some fixed $n \in \mathbb{N}$, then $\nabla(r)\otimes\Delta(s)$ is tilting.

\end{lemma}
\begin{proof}

First we note that we can apply \Cref{Martin} equally well to $E\otimes\Delta(r)$ since $(E\otimes\nabla(r))^* = E\otimes\Delta(r)$. Now we consider $\nabla(np + 1)\otimes E \otimes\Delta(np)$. Since the tensor product of tilting modules is also tilting (this follows from \cite[Theorem~1]{Mathieu}), this is a tilting module. Furthermore $p$ does not divide $np + 1$, so using the above result we obtain 
$$\nabla(np + 1)\otimes (E\otimes\Delta(np)) = (\nabla(np + 1)\otimes\Delta(np - 1)) \oplus (\nabla(np + 1)\otimes\Delta(np + 1)). $$

Since the whole module is tilting, each summand on the right hand side is tilting, in particular the module $\nabla(np + 1)\otimes\Delta(np - 1)$. We can continue to propagate in this manner, tensoring $E$ with both $\nabla(t)$ and $\Delta(t)$ for $t \in \{np,\ np + 1, \ldots,\ (n+1)p - 2\}$, until we reach $np-1$ and $(n+1)p - 1$, for which we can no longer apply the above result.

\end{proof}

This result shows us that there are more tilting modules of the form $\nabla(r)\otimes\Delta(s)$ than those given in \cite[Lemma~3.3]{Parker} for every characteristic $p$. Before we delve further into this investigation, we prove a couple of useful lemmas concerning tilting modules. For the following lemma, $G$ may be an arbitrary semisimple, simply connected algebraic group, over an algebraically closed field of prime characteristic. We will denote by $(\ ,\ )$ the usual positive definite symmetric bilinear form on the Euclidean space in which the root system of $G$ lies.
\begin{lemma}\label{tiltingTwist}
	Let $T_1$ and $T_2$ be tilting modules where $T_1$ is projective as a $G_1$-module, then the tensor product $T_1 \otimes T_2^F$ is also a tilting module.
\end{lemma}
\begin{proof}

First notice that it's sufficient to prove that for any $T_1$ and $T_2$ satisfying the hypothesis, the tensor product $T_1 \otimes T_2^F$ has a $\nabla$-filtration. Then the dual module $T_1^{*}\otimes (T_2^F)^{*}$ (a tensor product of tilting modules still satisfying the hypothesis) also has a $\nabla$-filtration, or equivalently, $T_1 \otimes T_2^F$ has a $\Delta$-filtration.
\\
\\
With this in mind let $\lambda \in X^{+}$, be such that $(\lambda,\check{\alpha}) \geq p - 1$ for all simple roots $\alpha$ (where, as usual $\check{\alpha} = 2\alpha/(\alpha,\alpha)$), so that $T(\lambda)$ is projective as a $G_1$-module \cite[Proposition~2.4]{Donkin}. Let $\rho$ be the half sum of all positive roots, then since $(\lambda - (p-1)\rho, \check{\alpha}) \geq 0$ for all simple roots $\alpha$, we have that $\lambda - (p-1)\rho \in X^{+}$. Let $\text{St}$ denote the Steinberg module $T((p-1)\rho)$, then $T(\lambda)$ must be a component of the module $\text{St}\otimes T(\lambda - (p-1)\rho)$ with highest weight $\lambda$. It follows that for any tilting module $T_2$ we have that $T(\lambda)\otimes T_2^F$ is a component of $\text{St}\otimes T(\lambda - (p-1)\rho)\otimes T_2^F$.
\\
\\
As a tilting module, $T_2$ has a $\nabla$-filtration, say with sections $\nabla(\mu(i))$ for some $\mu(i) \in X^{+}$. It follows that $\text{St}\otimes T_2^F$ has a $\nabla$-filtration with sections $\nabla((p-1)\rho + p\mu(i))$ (\cite[Proposition II.3.19]{Jantzen}). Hence $T(\lambda)\otimes T_2^F$ is a direct summand of a module with a $\nabla$-filtration, and thus has a $\nabla$-filtration itself.
\\
\\
Now suppose that $T_1$ is a tilting module that is projective as a $G_1$-module. Then each indecomposable summand of $T_1$ must be $T(\lambda)$ for some $\lambda$ as above (again using \cite[Proposition~2.4]{Donkin}). Hence $T_1$ is a direct sum of modules with a $\nabla$-filtration, and thus itself has a $\nabla$-filtration.
\end{proof}

We will use this lemma throughout the article, in conjunction with the facts that\\ $\nabla(p-1) = \Delta(p-1)$ is a projective $G_1$-module \cite[Proposition~II.10.1]{Jantzen}, and that the tensor product of a projective $G_1$-module with another $G_1$-module is again projective. Next we return to the case $G = SL_2(k)$.

\begin{lemma}\label{tiltingLemma}
	Let $V$ be a tilting module, and define the module $W$ by $H^0(G_1, V) = W^F$. Then $W$ is a tilting module.
\end{lemma}
\begin{proof}
	Since each tilting module has a unique decomposition (up to isomorphism) into indecomposable tilting modules $T(m)$ with highest weight $m$, it suffices to prove this for $V = T(m)$. We can split this into three separate cases, the first of which deals with $0 \leq m \leq p-1$. For such $m$ we have $T(m) = L(m)$ and so
	$$H^0(G_1, T(m)) = \begin{cases}
								L(0),  	&	\ m = 0 \\
								\\
								0, 	& \ 1 \leq m \leq p-1.
								\end{cases}$$
Next we consider the case $m = p-1 + t$ for $1 \leq t \leq p-1$. Here $T(m)$, considered as a $G_1$-module, is the injective envelope of $L(p-1-t)$ \cite[Example~2.2.1]{Donkin}. In particular $L(p-1-t)$ is the socle of $T(p-1+t)$ so if $H^0(G_1, T(p-1+t)) \neq 0$ then $H^0(G_1, L(p-1-t)) \neq 0$. Considering the case $t = p-1$ separately we get
$$H^0(G_1, T(m)) = \begin{cases}
								L(0),  	&	\ t = p-1 \\
								\\
								0, 	& \ 1 \leq t \leq p-2.
								\end{cases}$$
For the remaining cases we will use induction by writing $m = p-1 + t + pn$ for some $n\in\mathbb{N}$ and $0 \leq t \leq p-1$ so that we can write $T(m) = T(p-1+t)\otimes T(n)^F$. Taking the $G_1$ fixed points we get $H^0(G_1, T(m)) = H^0(G_1, T(p-1+t))\otimes T(n)^F$ which by the previous case gives us
\begin{align*}
 H^0(G_1, T(m)) &= \begin{cases}
								T(n)^F,  	&	\ t = p-1 \\
								\\
								0, 	& \ 0 \leq t \leq p-2,
								\end{cases} \\
\intertext{so that}
W &= \begin{cases}
								T(n),  	&	\ t = p-1 \\
								\\
								0, 	& \ 0 \leq t \leq p-2,
								\end{cases}								
\end{align*}
and is thus tilting.
\end{proof}

\section{Main Theorem}
Before stating the main theorem of this paper, we will introduce some notation. Let $r \in \mathbb{N}$ and $p$ a prime. We can write the base $p$ expansion of $r$ as

$$ \sum_{i = 0}^{n} r_i p ^ i,$$

where each $r_i \in \{0,\ldots, p-1\}$, $r_n \neq 0$ and for all $j > n$ we have $r_j = 0$. We will say that $r$ has $p$-length $n$ (or just length $n$ if the prime is clear), and write 

$$ \text{len}_p(r) = n.$$

We define $\text{len}_p(0) = -1$. Now given any pair $(r,s) \in \mathbb{N}^2$ we can write

$$r = \sum_{i = 0}^{n} r_i p ^ i, \quad s = \sum_{i=0}^{n}s_i p ^ i$$

where $n = \text{max}\,(\text{len}_p(r), \text{len}_p(s))$ so that at least one of $r_n$ and $s_n$ is non zero. Now let $m$ be the largest integer such that $r_m \neq s_m$ and let

$$ \hat{r} = \sum_{i = 0}^{m} r_i p ^ i, \quad \hat{s} = \sum_{i=0}^{m}s_i p ^ i$$

so that if $r>s$ we have $r_m > s_m$ and $\hat{r} > \hat{s}$. Using this notation we may write

$$ r = \hat{r} + \sum_{i = m+1}^{n} r_i p^i, \quad s = \hat{s} + \sum_{i = m+1}^{n} s_i p^i = \hat{s} + \sum_{i = m+1}^{n} r_i p^i.$$

Notice in particular that $r - \hat{r} = s - \hat{s}$ and denote this number by $\varepsilon_p(r,s)$ so that $p^{m+1}$ divides $\varepsilon_p(r,s)$. We will call the pair $(\hat{r}, \hat{s})$ the primitive of $(r,s)$, and say that $(r,s)$ is a primitive pair if $(r,s) = (\hat{r}, \hat{s})$.

\begin{lemma}\label{hatsLemma}

For $r$ and $s$ as above with $(r,s) \neq (\hat{r}, \hat{s})$, write $r = pt + r_0$ and $s = pu + s_0$, then we have the following.
\begin{enumerate}
	\item $\hat{r} = p\hat{t} + r_0$ and $\hat{s} = p\hat{u} + s_0$.
	\item $\varepsilon_p(r,s) = p\,\varepsilon_p(t, u)$.
	\item Let $r \geq s$ and $s' = s-1$. Then the pair $(\hat{r}, \hat{s'})$ is equal to $(\hat{r},\hat{s} - 1)$.
\end{enumerate}
\end{lemma}
\begin{proof}
	First we note that we have
	
	$$ t = \sum_{i = 0}^{n-1}r_{i+1}p^i, \quad u = \sum_{i=0}^{n-1}s_{i+1}p^i.$$
	
	Writing $t = \hat{t} + \varepsilon_p(t,u)$ and $u = \hat{u} + \varepsilon_p(t,u)$ it's clear that since $r_i = s_i$ for all $i = m+1, m+2, \ldots, n$ we have that 
	
	$$\hat{t} = \sum_{i= 0}^{m-1} r_{i+1}p^i, \quad \hat{u} = \sum_{i= 0}^{m-1} s_{i+1}p^i.$$
	
	It's clear then that $p\hat{t} + r_0 = \hat{r}$, and $p\hat{u} + s_0 = \hat{s}$. Furthermore we have that 
	
	$$\varepsilon_p(t,u) = \sum_{i = m}^{n-1}r_{i+1}p^i = \sum_{i = m+1}^n r_i p^{i-1},$$
	
	so that $p\varepsilon_p(t,u) = \varepsilon_p(r,s)$.
	\\
	\\
	For the final statement, we first note that since $r \geq s$ we must have that $r_m > s_m$, and since $s' < s$ it follows that $r_m > s'_m$. Hence we have that $\varepsilon_p(r,s) = \varepsilon_p(r, s')$. Now $s = \hat{s} + \varepsilon_p(r,s)$, so it follows that $s' = \hat{s} - 1 + \varepsilon_p(r,s)$. On the other hand we have  
	$$s' = \hat{s'} + \varepsilon_p(r,s') =  \hat{s'} + \varepsilon_p(r,s).$$
	Hence the pair $(\hat{r}, \hat{s'})$ is equal to $(\hat{r}, \hat{s} - 1)$.
\end{proof}

This lemma will be helpful in proving the main theorem, which follows.

\begin{theorem}\label{mainTheorem}
Let the pair $(\hat{r}, \hat{s})$ be the primitive of $(r,s)$. The module $\nabla(r) \otimes \Delta(s)$ is a tilting module if and only if one of the following
\begin{enumerate}
	\item $\hat{r} = a p^n + p^n - 1$ for some $a \in \{0,\ldots, p-2\}$, $n \in \mathbb{N}$, and $\hat{s} < p^{n+1}$,
	\item $\hat{s} = b p^n + p^n - 1$ for some $b \in \{0,\ldots, p-2\}$, $n \in \mathbb{N}$, and $\hat{r} < p^{n+1}$.
\end{enumerate}
\end{theorem}

The following picture illustrates which of the $\nabla(r) \otimes \Delta(s)$ are tilting up to $r,s \leq 31$ for $p=2$
.
\begin{center}
	\includegraphics[scale=0.74]{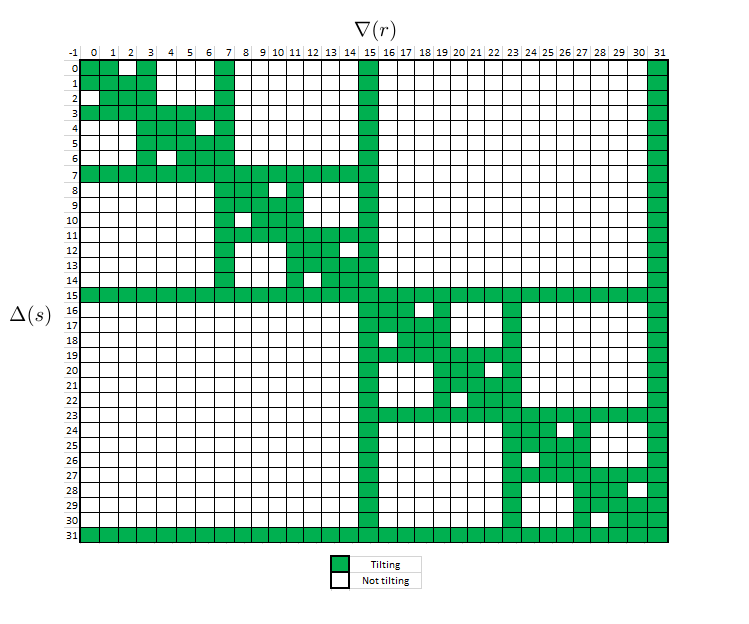}
	\captionof{figure}{The modules $\nabla(r)\otimes\Delta(s)$ when $\text{char}(k) = 2$.}
\end{center}

Before beginning the proof of this theorem, we will say a quick word on the figure above. It's clear that the pairs $(r,s)$ for which $(\hat{r},\hat{s}) = (r,s)$ are given by the intervals $[2^n, 2^{n+1}-1] \times [0, 2^n - 1]$ for $r \geq s$, and vice versa for $s \geq r$. Of these, according to the theorem, only the modules $\nabla(2^{n+1} - 1)\otimes\Delta(s)$ and $\nabla(r)\otimes\Delta(2^n - 1)$ are tilting (again, with $r \geq s$). With the mirrored situation for $s \geq r$, this accounts for the horizontal and vertical green lines appearing every $2^n - 1$.
\\
\\
For $(r,s)$ in the interval $[2^n, 2^{n+1} - 1]\times [2^n, 2^{n+1} - 1]$, we can determine whether or not $\nabla(r)\otimes\Delta(s)$ is tilting by looking (inductively) at the module $\nabla(r - 2^n)\otimes\Delta(s - 2^n)$. So, for example, in the block $[32, 63]\times[32,63]$ we would find a copy of the above figure. For the blocks $[32, 63]\times[0,31]$, we have only $\nabla(63)\otimes\Delta(s)$ and $\nabla(r)\otimes\Delta(31)$ are tilting, since in these blocks we have $(r,s) = (\hat{r},\hat{s})$. Similarly for block $[0,31]\times[32,63]$. 
\\
\\
In general we have that each block $[0, 2^n-1]\times[0, 2^n-1]$ can be split into four distinct blocks, each of size $2^{n-1}\times 2^{n-1}$. The two diagonal blocks are identical, and given by the block $[0, 2^{n-1}-1]\times[0, 2^{n-1}-1]$. The off diagonal blocks are those where $(\hat{r},\hat{s}) = (r,s)$.
\\
\\
To prove the theorem, we will first gather some elementary results on the modules $\nabla(r)\otimes\Delta(s)$.

\subsection{Lemmas}
\begin{lemma}\label{p-1}
Let $t,u \in \mathbb{N}$. The module $\nabla(pt + (p-1))\otimes\Delta(pu + (p-1))$ is tilting if and only if the module $\nabla(t)\otimes\Delta(u)$ is tilting.
\end{lemma}

\begin{proof}
	First recall the identities $\nabla(pt + (p-1)) = \nabla(p-1)\otimes\nabla(t)^F$ and $\Delta(pu + (p-1)) = \Delta(p-1)\otimes\Delta(u)^F$, found in \cite[Proposition~II.3.19]{Jantzen}. Using these we may rewrite $\nabla(pt + (p-1))\otimes\Delta(pu + (p-1))$ as $\nabla(p-1)\otimes\Delta(p-1)\otimes(\nabla(t)\otimes\Delta(u))^F$.
\\
\\
Using \Cref{tiltingLemma} we easily obtain the forward implication. The reverse implication is also clear since $\nabla(p-1)\otimes\Delta(p-1)$ is tilting and projective as a $G_1$-module, so we can apply \Cref{tiltingTwist}
\end{proof}

\begin{lemma}\label{odd prime tiltings}
Let $r = pt + v$, $s = pu + (p-1)$ for some $ 0 \leq v \leq p-2$ and $t,u \in \mathbb{N}$. Then $\nabla(r)\otimes\Delta(s)$ is tilting if and only if both $\nabla(t)\otimes\Delta(u)$ and $\nabla(t-1)\otimes\Delta(u)$ are tilting.
\end{lemma}

\begin{proof}

We will use the identity $\Delta(s) = \Delta(p-1)\otimes\Delta(u)^F$ as above, and the short exact sequence

$$0 \longrightarrow \nabla(v)\otimes\nabla(t)^F \longrightarrow \nabla(r) \longrightarrow \nabla(p-2-v)\otimes\nabla(t-1)^F \longrightarrow 0$$

which can be found in \cite[Satz~3.8, Bemerkung~2]{Jantzen1980}, in its dual form for Weyl modules. Notice that in the case $p=2$ we have that $v = 0 = p-2-v$, so this reduces to the sequence

$$0 \longrightarrow \nabla(t)^F \longrightarrow \nabla(r) \longrightarrow \nabla(t-1)^F \longrightarrow 0.$$

Tensoring the former with the latter gives the following short exact sequence
\begin{align*}
0 	\longrightarrow \nabla(v)\otimes&\Delta(p-1)\otimes(\nabla(t)\otimes\Delta(u))^F 
	\longrightarrow \nabla(r)\otimes\Delta(s)  \\
					&\longrightarrow \nabla(p-2-v)\otimes\Delta(p-1)\otimes(\nabla(t-1)\otimes\Delta(u))^F \longrightarrow 0.
\end{align*}
Since both $\nabla(v)\otimes\Delta(p-1)$ and $\nabla(p-2-v)\otimes\Delta(p-1)$ are tilting and projective as $G_1$-modules, if both $\nabla(t)\otimes\Delta(u)$ and $\nabla(t-1)\otimes\Delta(u)$ are also tilting then we have that $\nabla(r)\otimes\Delta(s)$ is an extension of tilting modules. The only such extensions are split (e.g. by \cite[Proposition~II.4.16]{Jantzen}), so we obtain $\nabla(r)\otimes\Delta(s)$ as a direct sum of two tilting modules, and hence is tilting itself.
\\
\\
For the converse statement we make the following observation: If $\nabla(r)\otimes\Delta(s) = \nabla(v + pt)\otimes\Delta(s)$ is tilting for some $v \in \{0, 1, \ldots, p-2\}$, then each module $\nabla(v' + pt)\otimes\Delta(s)$ for $v' \in\{0, 1, \ldots, p-2, p-1\}$ and the module $\nabla((p-1) + p(t-1))\otimes\Delta(s)$ are tilting too. This follows by repeated application of \Cref{Martin} by considering the tilting module $(\nabla(v + pt)\otimes E)\otimes\Delta(s)$ (as in \Cref{tensorE}). The result now follows from \Cref{p-1}. 

\end{proof}

It remains to determine which of the modules $\nabla(r)\otimes\Delta(s)$ are tilting when neither $r$ nor $s$ is congruent to $p-1$ modulo $p$. It turns out that this only occurs in the cases given in \Cref{tensorE}. In order to show this we will make use of the character.

\begin{lemma}\label{steinbergCharGeneral}
	Let $G$ be a semisimple, simply connected algebraic group over $k$, and let $T$ be a $G$-module that is projective as a $G_1$-module. Then $\chi((p-1)\rho)$ divides $\text{Ch}(T)$ in $\mathbb{Z}[\chi(t)]^W$.
\end{lemma}
\begin{proof}

This follows immediately from \cite[1.2(2)]{DonkinBundles}, since $T$ must also be a projective $B_1$ module.

\end{proof}

We now revert to the case $G = SL_2(k)$ and obtain the following corollary.

\begin{corollary} \label{steinbergChar}
For all $r \geq p-1$, the character of the Steinberg module $\nabla(p-1)$ divides that of the indecomposable tilting module $T(r)$ of highest weight $r$. \qed
\end{corollary}

%
%
%
%

Now let's consider the character $\chi(r) \in \mathbb{Z}[x, x^{-1}]$. We have that
\begin{center}
\begin{align*}
	\chi(r) = & \ x^r + x^{r-2} + \ldots + x^0 + \ldots + x^{-r}\\
			= & \ \frac{1}{x^r}(x^{2r} + x^{2r - 2} + \ldots + 1)\\
			= & \ \frac{1}{x^r} \bigg( \frac{x^{2r+2} - 1}{x^2 - 1} \bigg),
\end{align*}
\end{center}
so the roots of this equation are the $(2r+2)^{\text{th}}$ roots of unity, except $\pm 1$. If $\chi(p-1)$ divides $\chi(r)$ then, we must have the $2p^{\text{th}}$ roots of unity are also $(2r+2)^{\text{th}}$ roots of unity, which would imply that $p$ divides $r+1$, i.e. that $r$ is congruent to $p-1$ modulo $p$.
\\
\\
Hence we have shown that if both $r$ and $s$ are not congruent to $p-1$ modulo $p$, the character $\chi(p-1)$ does not divide $\text{Ch}(\nabla(r)\otimes\Delta(s)) = \chi(r)\chi(s)$. Now suppose that $\nabla(r)\otimes\Delta(s)$ is tilting, and that $|r-s| > p-1$. By considering its good filtration, we see that the decomposition of $\nabla(r)\otimes\Delta(s)$ into indecomposable tilting modules cannot contain any $T(j)$ for $j =0, \ldots, p-1$. By \Cref{steinbergChar} its character is divisible by $\chi(p-1)$ but the above calculation contradicts this. In summary:

\begin{lemma} \label{notTilting}
	For $r$ and $s$ both not congruent to $p-1$ modulo $p$, and $|r-s| > p-1$, the module $\nabla(r)\otimes\Delta(s)$ is not tilting. \qed
\end{lemma}

There are now only a few more cases which we have not considered. These occur when $|r-s| \leq p-1$, but not both of $r$ and $s$ lie in the set given in \Cref{tensorE} (for example, take $r = np$ and $s = np-2$). We can swiftly deal with these cases, by once again appealing to \Cref{Martin}, but we must first make precise exactly which $r$ and $s$ we are considering. We will assume that $r > s$, but the argument works equally well for $r < s$.
\\
\\
Since $|r-s| \leq p-1$ and at least one of $r$ and $s$ is not in the set  $\{np - 1,\ np,\ np + 1, \ldots,\ (n+1)p - 1\}$ for all $n \in \mathbb{N}$ we must have that $r \in \{np, np + 1, \ldots, (n+1)p - 2\}$ and $s \in \{(n-1)p, \ldots, np - 2\}$ for some fixed $n$. Now if $\nabla(r)\otimes\Delta(s)$ is tilting, then by applying \Cref{Martin} (as in \Cref{tensorE}) we obtain that for all $r' \in \{np, np + 1, \ldots, (n+1)p - 2\}$ and $s' \in \{(n-1)p, \ldots np - 2\}$, the module $\nabla(r')\otimes\Delta(s')$ is tilting. Taking $r' = (n+1)p - 2$ and $s' = (n-1)p$ however, contradicts \Cref{notTilting}, and so $\nabla(r)\otimes\Delta(s)$ is not tilting.

\begin{lemma} \label{notTilting2}
	For $r$ and $s$ both not congruent to $p-1$ modulo $p$, and not both in the set $\{np - 1,\ np,\ np + 1, \ldots,\ (n+1)p - 1\}$ for any $n \in \mathbb{N}$, the module $\nabla(r)\otimes\Delta(s)$ is not tilting. \qed
\end{lemma}

We have now determined exactly which of the modules $\nabla(r)\otimes\Delta(s)$ are tilting for all primes $p$ (recall that for $r, s \in \{ 0,1,\ldots,p-1\}$ the module $\nabla(r)\otimes\Delta(s)$ is tilting, so we can begin applying \Cref{odd prime tiltings} to these modules). The figure below illustrates this for $p=3$.
\\
\\
The pattern here is similar to that for the case $p=2$, however we now have the extra complication that the values $a$ and $b$ from the theorem can be either $0$ or $1$ (whereas in the $p=2$ case, we had that $a = b = 0$), and the coefficients in the base $3$ expansion are in $\{0,1,2\}$. As such, we have that for a given $n\in \mathbb{N}$, those blocks where $(r,s) = (\hat{r},\hat{s})$ can be broken up into the union 

$$[3^n, 2\times 3^n - 1]\times[0, 3^n - 1]\cup[2\times 3^n, 3^{n+1}-1]\times[0, 3^n - 1]\cup [2\times 3^n, 3^{n+1}-1]\times[3^n, 2\times 3^n - 1],$$

in the case $r \geq s$. The result of this is that each interval $[3^n, 3^{n+1} - 1]$ is split into two at $2\times 3^n - 1$, and so each block $[0, 3^n - 1]\times[0, 3^n - 1]$ is split into nine distinct $3^{n-1}\times 3^{n-1}$ sections, with the three on the diagonal given inductively, as in the characteristic $2$ case.

\begin{center}
	\includegraphics[scale=0.75]{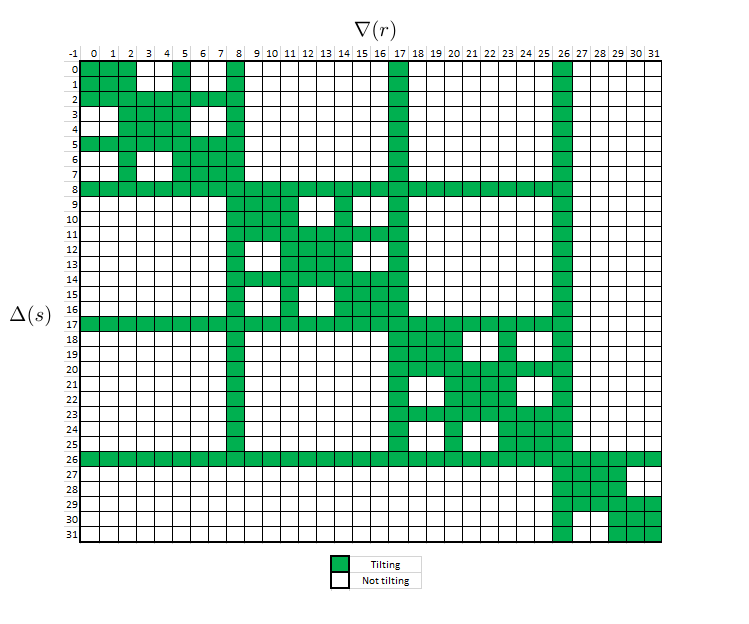}
	\captionof{figure}{The modules $\nabla(r)\otimes\Delta(s)$ when $\text{char}(k) = 3$.}
\end{center}

\subsection{Proof of \Cref{mainTheorem}}

We are now ready to prove \Cref{mainTheorem}, which we will do in two steps. The first is to show that for a primitive pair $(\hat{r},\hat{s})$, we have that $\nabla(\hat{r})\otimes\Delta(\hat{s})$ is a tilting module if and only if $\hat{r}$ and $\hat{s}$ are as described in the statement of the theorem. The second step is to show that for any pair $(r,s)$ with primitive pair $(\hat{r},\hat{s})$, we have that $\nabla(r)\otimes\Delta(s)$ is tilting if and only if $\nabla(\hat{r})\otimes\Delta(\hat{s})$ is tilting.
\\
\\
We first note that since the dual of a tilting module is also a tilting module, and we have the relation $(\nabla(r)\otimes\Delta(s))^{*} = \nabla(s)\otimes\Delta(r)$, it will be safe to assume that $r \geq s$, and simply take the dual for the case $r < s$.
\\
\\
1.) We will begin the first step by assuming that, for a primitive pair $(r,s)$, we have that $\nabla(r)\otimes\Delta(s)$ is tilting. We want to show that this forces $r$ and $s$ to be of the form in the statement of the theorem. We will assume that $r \geq s$ and proceed by induction on $\text{len}_p(r) = N$. For $N=0$ we have that $r \leq p-1$ and so $r = ap^N + p^N - 1$ for $a = 0,\ldots, p-2$, or in the case $r = p-1$ we have $r = p^{N+1} -  1$. In each case we have that $r$ is of the desired form, and $s < p^{N+1}$.
\\
\\
Next let's write $r = pt + r_0$ and $s = pu + s_0$, where $\text{len}_p(t) = \text{len}_p(r) - 1$, and similarly for $s$ and $u$. Now, if $\nabla(r)\otimes\Delta(s)$ is tilting, by \Cref{notTilting2} we must have that either $r_0$ or $s_0$ is equal to $p-1$, or $r$ and $s$ both lie in the set $\{ np-1, np, np + 1, \ldots, (n+1)p -1\}$. In the second case, since we are assuming the pair $(r,s)$ to be primitive, we must have that one of $r$ or $s$ is equal to $np - 1$, and the other is not. In any case then, we may assume then $r_0 = p-1$ so that $r = pt + p - 1$. 
\\
\\
Now we have two cases to consider, the first is that $s_0 = p-1$, and the second that $s_0 \neq p - 1$. Let's suppose that $s_0 = p -1$, then by \Cref{p-1} we have that $\nabla(t)\otimes\Delta(u)$ is tilting. By induction we must have that $t$ and $u$ are of the form given in the statement of the theorem. Without loss of generality, we may assume $t = ap^{N-1} + p^{N-1} - 1$ for some $a \in \{0,\ldots,p-2\}$, and $u \leq p^N-1$. Hence we have that 

$$ r = p(ap^{N-1} + p^{N-1} - 1) + p - 1 = ap^N + p^N - 1,$$

and $s \leq p^{N} - p + s_0$, which is strictly less than $p^{N+1}$ since $s_0 < p$.
\\
\\
For the second case, we suppose that $s_0 \neq p - 1$ (the case $r_0 \neq p-1$ and $s_0 = p-1$ is obtained in the same manner), so that by \Cref{odd prime tiltings} we have that $\nabla(t)\otimes\Delta(u)$  and $\nabla(t)\otimes\Delta(u-1)$ are tilting By induction we have that the pairs $(t, u)$ and $(t, u-1)$ are both of the form in the theorem. Since we cannot have that both $u$ and $u-1$ are of the form $ap^{N-1} + p^{N-1} - 1$, we must have that $t$ is of this form, and we complete the proof as above.
\\
\\
Now we prove the converse statement, that is, if $r = a p^n + p^n - 1$ for some $a \in \{0,\ldots, p-2\}$, $n \in \mathbb{N}$, and $s < p^{n+1}$, then $\nabla(r)\otimes\Delta(s)$ is tilting. Once again, we will use induction on $n$, with the case $n = 0$ being clear. For the inductive step, we have that if $r = pt + p-1$ and $s = pu + s_0$ then $t = ap^{n-1} + p^{n-1} - 1$ and $u < p^n$. Then by induction the modules $\nabla(t)\otimes\Delta(u)$ and $\nabla(t)\otimes\Delta(u-1)$ are tilting, so by either \Cref{p-1} or \Cref{odd prime tiltings} we have that $\nabla(r)\otimes\Delta(s)$ is tilting too.
\\
\\
2.) For the next step we prove the statement: $\nabla(r)\otimes\Delta(s)$ is tilting if and only if $\nabla(\hat{r})\otimes\Delta(\hat{s})$ is tilting. First, let's assume that $\nabla(r)\otimes\Delta(s)$ is tilting. By \Cref{notTilting2} we have that either one of $r$ and $s$ is congruent to $p-1$ modulo $p$ or they lie in the set $\{ np-1, np, np + 1, \ldots, (n+1)p -1\}$ for some $n\in \mathbb{N}$. Suppose that both $r$ and $s$ lie in the set $\{ np-1, np, np + 1, \ldots, (n+1)p -1\}$. If neither are equal to $np - 1$ then it's clear that $\hat{r}$ and $\hat{s}$ lie in the set $\{0,\ldots,p-1\}$, and so $\nabla(\hat{r})\otimes\Delta(\hat{s})$ is tilting. Note that  if we assume $\nabla(\hat{r})\otimes\Delta(\hat{s})$ is tilting, we must also have that either $r$ or $s$ is congruent to $p-1$ modulo $p$, as in step 1. We may then, only consider the case that at least one of $r$ and $s$ is congruent to $p-1$ modulo $p$.
\\
\\
Let's suppose that $r = pt + p-1$ and $s = pu + s_0$, so that by \Cref{hatsLemma} we have $\hat{r} = p\hat{t} + p - 1$ and $\hat{s} = p\hat{u} + s_0$. As in the previous step, there are two cases to consider: $s_0 = p-1$ and $s_0 \neq p-1$. In both cases we will proceed by induction on $\text{len}_p(r)$ with $r \geq s$.
\\
\\
First, we assume $s_0 = p-1$. For the base case $\text{len}_p(r) = 0$ we have that $r = s = p-1$ so $(\hat{r},\hat{s}) = (0,0)$, and the result is clear. Now by \Cref{p-1} we have $\nabla(r)\otimes\Delta(s)$ is tilting if and only if $\nabla(t)\otimes\Delta(u)$ is tilting. By induction then we have that this is tilting if and only if $\nabla(\hat{t})\otimes\Delta(\hat{u})$ is tilting, so applying \Cref{p-1} again (since $\hat{r} = p\hat{t} + p - 1$ and $\hat{s} = p\hat{u} + p-1$) we find that this is if and only if $\nabla(\hat{r})\otimes\Delta(\hat{s})$ is tilting.
\\
\\
Next, we assume that $s_0 \neq p-1$. Again, the base case is easily obtained since this time the pair $(p-1, s_0)$ is primitive. Now, for the inductive step we have, by \Cref{odd prime tiltings}, that $\nabla(r)\otimes\Delta(s)$ is tilting if and only if both $\nabla(t)\otimes\Delta(u)$ and $\nabla(t)\otimes\Delta(u-1)$ are tilting. By induction we have that these are tilting if and only if $\nabla(\hat{t})\otimes\Delta(\hat{u})$ and $\nabla(\hat{t})\otimes\Delta(\hat{u-1})$ are tilting. Now since $(\hat{t}, \hat{u-1}) = (\hat{t}, \hat{u} - 1)$ (\Cref{hatsLemma}), we apply \Cref{odd prime tiltings} again to obtain that this is if and only if $\nabla(\hat{r})\otimes\Delta(\hat{s})$ is tilting.\qed

\section*{Acknowledgements}

I wish to thank Stephen Donkin for his help and guidance given throughout this project, and acknowledge financial support from EPSRC.

\bibliography{tiltingbib}

\begin{thebibliography}{1}

\bibitem{Donkin}
S.~Donkin.
\newblock On tilting modules for algebraic groups.
\newblock {\em Mathematische Zeitschrift}, 212:39--60, 1993.

\bibitem{DonkinBundles}
S.~Donkin.
\newblock The cohomology of line bundles on the three-dimensional flag variety.
\newblock {\em Journal of Algebra}, 307:570--613, 2007.

\bibitem{Doty-Henke}
S.~Doty and A.~Henke.
\newblock Decomposition of tensor products of modular irreducibles for
  $\text{SL}_2$.
\newblock {\em The Quarterly Journal of Mathematics}, 56:189--207, 2005.

\bibitem{Humphreys-Jantzen}
J.E. Humphreys and J.C. Jantzen.
\newblock Blocks and indecomposable modules for semisimple algebraic groups.
\newblock {\em Journal of Algebra}, 54:494--503, 1978.

\bibitem{Jantzen1980}
J.C. Jantzen.
\newblock {D}arstellungen halbeinfacher {G}ruppen und ihrer
  {F}robenius-{K}erne.
\newblock {\em Journal für die reine und angewandte Mathematik}, 317:157--199,
  1980.

\bibitem{Jantzen}
J.C. Jantzen.
\newblock {\em Representations of Algebraic Groups}.
\newblock Academic Press, 1987.

\bibitem{Mathieu}
O.~Mathieu.
\newblock Filtrations of ${G}$-modules.
\newblock {\em Annales scientifiques de l'É.N.S.}, 23:625--644, 1990.

\bibitem{Parker}
A.~Parker.
\newblock Higher extensions between modules for $\text{SL}_2$.
\newblock {\em Advances in Mathematics}, 209:381--405, 2006.

\end{thebibliography}
\bibliographystyle{plain}

\end{document}